\documentclass[11pt,a4paper]{article}

\usepackage{mathrsfs}
\usepackage{amsmath}
\usepackage{amsfonts}
\usepackage{amssymb}
\usepackage{graphicx}
\usepackage{enumerate}
\usepackage{arydshln}
\usepackage[ruled]{algorithm2e}
\usepackage{appendix}
\usepackage{listings}
\usepackage{xcolor}
\usepackage{cite}

\usepackage{makeidx,epsf,psfrag,,epsfig,graphics}
\usepackage{graphicx,subfigure,epsfig}
\usepackage{tikz}
\usepackage[ruled]{algorithm2e}

\usepackage[colorlinks,linkcolor=blue, anchorcolor=blue, citecolor=blue]{hyperref}

\allowdisplaybreaks

\begin{document}

\newtheorem{theorem}{Theorem}[section]
\newtheorem{corollary}[theorem]{Corollary}
\newtheorem{definition}[theorem]{Definition}
\newtheorem{conjecture}[theorem]{Conjecture}
\newtheorem{question}[theorem]{Question}
\newtheorem{lemma}[theorem]{Lemma}
\newtheorem{fact}[theorem]{Fact}
\newtheorem{property}[theorem]{Property}
\newtheorem{proposition}[theorem]{Proposition}
\newtheorem{quest}[theorem]{Question}
\newtheorem{example}[theorem]{Example}
\newenvironment{proof}{\noindent {\bf
Proof.}}{\rule{2.5mm}{2.5mm}\par\medskip}
\newcommand{\remark}{\medskip\par\noindent {\bf Remark.~~}}
\newcommand{\pp}{{\it p.}}
\newcommand{\de}{\em}

\newcommand{\ec}{{\rm ecc}}
\newcommand{\Near}{{\rm Near}}
\renewcommand{\appendix}{\bf \Large Appendix}

\title{On the computational complexity of the Steiner $k$-eccentricity}


\author{Xingfu Li$^a$, Guihai Yu$^{a}$, Aleksandar Ili\' c$^b$, Sandi Klav\v{z}ar$^{c,d,e,\dag}$ \\ \\
{\small  $^{a}$College of Big Data Statistics, Guizhou University of Finance and Economics}\\
{\small Guiyang, Guizhou, 550025, China.}\\
 {\small E-mail: { \tt xingfuli@mail.gufe.edu.cn; yuguihai@126.com}}\\
 {\small $^b$Facebook Inc, Menlo Park 94025, California, USA}\\
 {\small E-mail: \tt{aleksandari@gmail.com}}\\
 {\small  $^c$ Faculty of Mathematics and Physics, University of Ljubljana, Slovenia}\\
{\small $^{d}$ Institute of Mathematics, Physics and Mechanics, Ljubljana, Slovenia} \\
{\small $^{e}$ Faculty of Natural Sciences and Mathematics, University of Maribor, Slovenia}\\
{\small {\tt sandi.klavzar@fmf.uni-lj.si}}}

\maketitle

\begin{abstract}
 The Steiner $k$-eccentricity of a vertex $v$ of a graph $G$ is the maximum Steiner distance over all $k$-subsets of $V (G)$ which contain $v$. A linear time algorithm for calculating the Steiner $k$-eccentricity of a vertex on block graphs is presented. For general graphs, an $O(n^{\nu(G)+1}(n(G) + m(G) + k))$ algorithm is designed, where $\nu(G)$ is the cyclomatic number of $G$.  A linear algorithm for computing the Steiner $3$-eccentricities of all vertices of a tree is also presented which improves the quadratic algorithm from [Discrete Appl.\ Math.\  304  (2021) 181--195].
\end{abstract}


\section{Introduction}

Every graph $G = (V(G), E(G))$ in this paper is simple and undirected. The order of $G$ will be denoted by $n(G)$ and the size of $G$ by $m(G)$. The {\em cyclomatic number}, $\nu(G)$, of $G$ is the minimum number of edges of $G$ whose removal makes $G$ acyclic. If $G$ is connected, then $\nu(G) = m(G) - n(G) + 1$. (The cyclomatic number of $G$ can alternatively be defined as the dimension of its cycle space.) If every block of $G$ is a clique, then $G$ is a \emph{block graph}. The \emph{distance} $d_{G}(u,v)$ between vertices $u$ and $v$ in $G$ is the length of a shortest $u,v$-path.  

The {\em eccentricity} $\ec_G(v)$ of a vertex $v$ in $G$ is the maximum distance between $v$ and all the other vertices of $G$. We refer to~\cite{Pei2019OnACandSub, Dankelmann2014, Dankelmann2019Osaye, Dankelmann2019AverageEK, Du2013, Du2016, Ili2012, Smith2016, Tang2012} for different investigations of the eccentricity. In this work we study a generalization of the eccentricity, the Steiner $k$-eccentricity. Its definition is based on Steiner trees which are in turn defined as follows. If $S\subseteq V(G)$, then a subgraph $T$ of $G$ is a {\em Steiner $S$-tree}, if $T$ is a minimum connected subgraph of $G$ which spans all vertices from $S$. Every vertex from $S$ is called a \emph{terminal} of $T$, and the set $S$ is the \emph{terminal set} of $T$. The {\em Steiner $k$-eccentricity}, $\ec_{k}(v,G)$, of a vertex $v$ is the maximum size over all Steiner $S$-trees, where $|S| = k$ and $v\in S$, that is, 
$$\ec_{k}(v,G) = \max_{S\subseteq V(G) \atop |S|=k,v\in S} \{ m(T):\ T\ {\rm is\ a\ Steiner}\ S\hbox{-}{\rm tree}\}\,.$$
(For additional aspects of the Steiner distance see~\cite{gologranc-2018, klavzar-2021, li-2016, mao-2021, weis-2020}.) Note that  $\ec_G(v) = \ec_{2}(v,G)$. A Steiner $S$-tree $T$ that realizes $\ec_{k}(v,G)$ is a \emph{Steiner $k$-eccentricity tree of $v$}, we will shortly say that $T$ is a  \emph{Steiner $k$-ecc $v$-tree}. The $k$-set $S$ corresponding the the Steiner $k$-ecc $v$-tree is a \emph{Steiner $k$-ecc $v$-set} in $G$. The problem to find a Steiner $k$-eccentricity tree of a given vertex is referred to as the \emph{Steiner $k$-eccentricity tree} (k-ST) problem. The decision version of the \emph{Steiner k-eccentricity tree problem} ($k$-ST) is presented in Table~\ref{k-STP}.

\begin{table}[ht!]
\centering  
	\caption{The Steiner $k$-eccentricity tree problem ($k$-ST)}  
	\label{k-STP}  
	\begin{tabular}{|l|}
		\hline
		Instance: Graph $G$, $v\in V(G)$, $k\in [n(G)]$, constant $c$.\\
		\hline
        Question: Is there a Steiner $S$-tree $T$, where $|S| = k$ and $v\in S$, such that $m(T)\ge c$? \\
        \hline
	\end{tabular}
\end{table}

The minimum Steiner tree problem is a well-known NP-hard problem~\cite{Garey1979}, but the hardness of the  $k$-ST problem is still unknown. In~\cite{li2020average}, a linear time algorithm was designed to find the optimal value of the 3-ST problem on trees, while in~\cite{li2020steiner} the result was extended to the $k$-ST problem. In the following section we design a linear time algorithm for the $k$-ST problem on block graphs. In the subsequent section we present an algorithm for the problem on general graphs with the time complexity $O(n^{\nu(G)+1}(n(G) + m(G) + k))$. In Section~\ref{all-ecc} we present a linear algorithm to calculate the Steiner $3$-eccentricity for all vertices in a weighted tree. This improves the corresponding quadratic algorithm for (unweighted) trees from~\cite{li2020average}. We conclude the paper by giving several directions for future work.

\section{A linear algorithm for block graphs}
\label{block-graphs}

In this section, we devise a linear-time algorithm to solve the $k$-ST problem on block graphs. The main idea is to reduce the problem from a block graph $G$ to a special spanning tree $T$ such that the equqality $\ec_{k}(v,G) = \ec_{k}(v,T)$ holds, and then to invoke the algorithm from~\cite{li2020steiner}. 

Let $G$ be a block graph. If $v\in V(G)$ and $B$ is a block of $G$, then let $\Near_G(v,B)$ be a nearest vertex to $v$ in the block $B$, see Fig.~\ref{E-tree} for an example. We first observe that $\Near_G(v,B)$ is unique. 

\begin{property}\label{unique-nerast-vertex}
If $G$ is a block graph, $v\in V(G)$, and $B$ a block of $G$, then $\Near_G(v,B)$ is unique.
\end{property}

\begin{proof}
If $v\in V(B)$, then clearly $\Near_G(v,B)=v$ is unique. If $v\notin V(B)$, then let $P$ be a shortest path between $v$ and the block $B$. Then the end-vertex $x$ of $P$, $x\ne v$, is a cut vertex of $B$ which in turn implies that  $\Near_G(v,B)=x$ is again unique. 
\end{proof}

Let $v$ be a vertex of a block graph $G$. For every block $B$ of $G$ remove every edge which is not incident with $\Near_G(v,B)$ and denote the resulting graph by $T(v,G)$, for an example see Fig.~\ref{E-tree} again. Considering shortest paths between $v$ and the cut vertices of $G$ we infer that $T(v,G)$ is connected. Moreover, it is also clear (having in mind that $G$ is a block graph) that $T(v,G)$ has no cycle. Hence $T(v,G)$ is a spanning tree of $G$. In addition, with Property~\ref{unique-nerast-vertex} in hands it immediately follows from the construction that $T(v,G)$ is unique for each vertex $v$ of a block graph $G$.  

\begin{figure}[t!]
\begin{center}
\epsfig{file=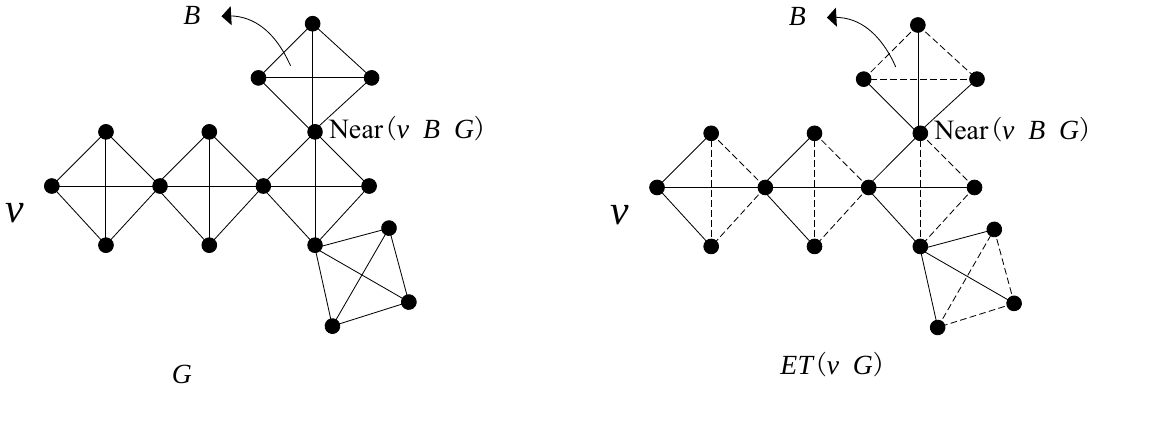, scale=1.2}
\end{center}
\vspace{-1.0cm}
\caption{Block graph $G$ (left) and $T(v,G)$ (right). The dashed lines indicate the edges which must be  removed from $G$ to construct $T(v,G)$.}
\label{E-tree}
\end{figure}

We are now ready for the key result needed for our algorithm for block graphs.

\begin{theorem}\label{maintain-optimal}
Let $v$ be a vertex of a block graph $G$. Then a Steiner $k$-ecc $v$-tree in $T(v,G)$ is also a Steiner $k$-ecc $v$-tree in $G$.
\end{theorem}

\begin{proof}
Let $T_{1}$ be a Steiner $k$-ecc $v$-tree in $T(v,G)$ and suppose on the contrary that $T_{1}$ is not a Steiner $k$-ecc $v$-tree in $G$.  Let $T_{2}$ be a Steiner $k$-ecc $v$-tree in $G$. Then we have $m(T_{1}) \not= m(T_{2})$. Moreover, since $T_1$ is also a subtree of $G$,  we must have $m(T_{1}) < m(T_{2})$. We are now going to construct a tree $T'_{2}$ of $T(v,G)$ with $m(T'_{2}) > m(T_{1})$, which will contradict the fact that $T_{1}$ is a Steiner $k$-ecc $v$-tree in $T(v,G)$.

Construct the tree $T'_{2}$ from $T_{2}$ through the following procedure. For every block $B$ of $G$, if there is an edge $e\in E(B)\cap E(T_{2})$ such that neither endpoint of $e$ is  the vertex $\Near_G(v,B)$, then delete the edge $e$ from $T_{2}$, and add an edge between one endpoint of $e$ and $\Near_G(v,B)$. After finishing the whole procedure for all blocks of $G$, the tree $T'_{2}$ is constructed. Since the edge deletion and addition occur pairwise, $m(T'_{2}) = m(T_{2})$. Since $m(T_{1}) < m(T_{2})$, we have the announced contradiction $m(T_{1}) < m(T'_{2})$. 
\end{proof}

Theorem~\ref{maintain-optimal} directly leads to Algorithm \ref{k-ecc-block}.

\medskip
\begin{algorithm}[H]\label{k-ecc-block}
\caption{k-ECC-Block($v$, $G$, $k$)}
\LinesNumbered 
\KwIn{Block graph $G$, vertex $v\in V(G)$, an integer $k\geq 3$. }
\KwOut{$\ec_k(v,G)$.}
Determine $T(v,G)$\; \label{find-T}
Return $\ec_k(v,T(v,G))$;\label{invoke-tree-alg} 
\end{algorithm}

\begin{theorem}\label{complexity-on-block-graphs}
If $G$ is a block graph and $v\in V(G)$, then Algorithm~\ref{k-ecc-block} computes $\ec_k(v,G)$ and can be implemented to run in $O(k(n(G) + m(G)))$ time. 
\end{theorem}

\begin{proof}
The correctness of the algorithm follows from Theorem~\ref{maintain-optimal}. 

Since $T = T(v,G)$ is a tree, Step \ref{invoke-tree-alg} can be implemented in $O(k(n(T) + m(T))$ time by invoking the corresponding algorithm from~\cite{li2020steiner}. As for Step~\ref{find-T}, to determine $T(v,G)$ efficiently, Algorithm~\ref{Get-Tree} modifies the depth-first search (DFS) algorithm, and runs in linear time. 
\end{proof}

\medskip
\begin{algorithm}[H]\label{Get-Tree}
\caption{Get-Tree($v, G$)}
\LinesNumbered 
\KwIn{Block graph $G$, vertex $v\in V(G)$.}
\KwOut{$T(v,G)$.}
Mark all vertices as '\emph{unvisited}', and mark the vertex $v$ as '\emph{visited}'\;
\For {each unvisited vertex $u\in N_{G}(v)$}{
    \For {each vertex $w\in N_{G}(u)$}{
        \If {$wv\in E(G)$}{
            Delete $uw$ from $G$\;
        }
    }
    Get-Tree($u, G$)\;
}
\Return $G$;
\end{algorithm}

\section{On general graphs}
\label{general-graphs}

The basic property that allows a fast algorithm for calculating the Steiner $k$-eccentricity of a vertex in a tree is that every Steiner $k$-ecc $v$-tree  contains a Steiner $(k-1)$-ecc $v$-tree~\cite{li2020steiner}. This property does not hold in general graphs. In fact, as the example from Fig.~\ref{counter-example} demonstrates, the property does not hold even on unicycle graphs. 

\begin{figure}[ht!]
\begin{center}
\epsfig{file=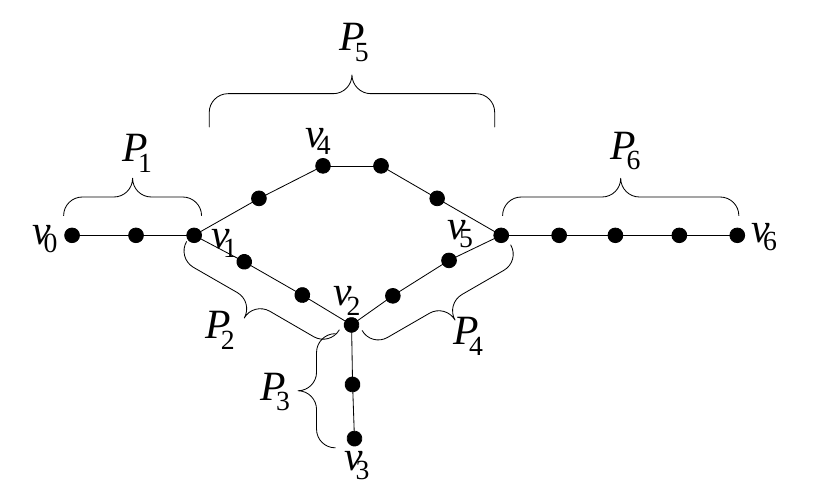, scale=1.0}
\end{center}
\vspace{-0.5cm}
\caption{The Steiner $3$-ecc tree of the vertex $v_{0}$ is formed by the paths $P_{1}$, $P_{2}$, $P_{3}$, $P_{4}$, and $P_{6}$. The Steiner $2$-ecc tree of $v_{0}$ is formed by the paths $P_{1}$, $P_{5}$ and $P_{6}$. Hence the Steiner $3$-ecc $v_{0}$-tree does not contain a Steiner $2$-ecc $v_{0}$-tree.}
\label{counter-example}
\end{figure}

This example illustrates that it  could be difficult to find a property that would lead to a fast algorithm for calculating the Steiner $k$-eccentricity of a vertex on general graphs. In this section, we present two algorithms to solve the $k$-ST problem in general graphs. The first one is a brute-forced method, the other one reduces the problem from general graphs to trees. The running time of the brute-forced method grows exponentially with $k$, while the running time of the latter algorithm grows exponentially in $\nu(G)$. 

\subsection{Brute-force algorithm}

The definition of the Steiner $k$-eccentricity of a vertex leads to a direct, brute-force algorithm as follows. Let $v$ be a vertex for which we are going to calculate the Steiner $k$-eccentricity in a graph $G$. Initially, enumerate all ($k-1$)-subsets $S$ in $V(G)\setminus \{v\}$. For each of these sets $S$ then invoke an algorithm to find a minimum Steiner tree for the set $S\cup\{v\}$. Finally, choose the maximum size among all these minimum Steiner trees. 

The Steiner tree problem is a well-known NP-hard problem~\cite{Garey1979}.   Erickson, Monma, and Veinott~\cite{Erickson1987SendandSplitMF} designed an exact algorithm for the Steiner tree problem with running time $O(3^{k}n+2^{k}(m+n\log n))$, where $n = n(G)$, $m = m(G)$, and $k$ is the number of terminals. Fuchs et al.~\cite{Fuchs2007DynamicPF} and Vygen~\cite{Vygen2011FasterAF} followed by exact algorithms with running times $O(2^{k+(k/2)^{1/3}(\ln n)^{2/3}})$ and $O(nk2^{k+\log_{2}k\log_{2}n})$, respectively. Note that the running time of all these algorithms grows exponentially with the number of terminals. So the running time of the  brute-forced method also grows exponentially with respect to the number of terminals.

\subsection{Reducing to trees}

We now devise a novel algorithm whose running time does not grow exponentially on the number of terminal, that is, on $k$. The key for the algorithm is the following result, where ${\cal T}(G)$ denotes the set of all spanning trees of a connected graph $G$. 

\begin{theorem}\label{ecc-on-graphs}
If $G$ is a connected graph, then 
\begin{equation}
\ec_{k}(v,G) = \min\{\ec_{k}(v,T):\ T\in {\cal T}(G)\}\,.
\end{equation}
\end{theorem}

\begin{proof}
Let $G$ be a graph, $v\in V(G)$, and $T$ s spanning tree of $G$. Then we claim that the size of a Steiner $k$-ecc $v$-tree in $G$ is not larger than the size of a Steiner $k$-ecc $v$-tree in $T$.

Suppose on the contrary that there is a Steiner $k$-ecc $v$-tree $T_{1}$ in $T$ such that $m(T_{1})$ is less than the size of a Steiner $k$-ecc $v$-tree in $G$. Let $T_{v}$ be a Steiner $k$-ecc $v$-tree in $G$ and $S_{v}$ be the corresponding Steiner $k$-ecc $v$-set. Then we have $m(T_{1}) < m(T_{v})$. Since $T$ is a spanning tree of $G$, we clearly have $S_{v}\subseteq V(T)$. Moreover, the size of the minimum Steiner tree on the set $S_{v}$ in $T$ is not less than the size of $T_{v}$. Let $T_{2}$ be a minimum Steiner tree on the set $S_{v}$ in $T$. Then we have $m(T_{v}) \leq m(T_{2})$ and therefore, $m(T_{1}) < m(T_{2})$. This contradicts to the assumption that $T_{1}$ is a Steiner $k$-ecc $v$-tree in $T$, hence the claim is proved.  

From the claim it now follows that the value of $\ec_{k}(v,G)$ is equal to $\min\{\ec_{k}(v,T):T\in {\cal T}(G)\}$.
\end{proof}

In order to determine the Steiner $k$-eccentricity of a vertex in a graph $G$, Theorem~\ref{ecc-on-graphs} says that it suffices to calculate the Steiner $k$-eccentricity of the vertex in every spanning tree. In our algorithm we  enumerate all possible edge sets of size $\nu(G)$ rather than enumerating all spanning trees. Moreover, we calculate the Steiner $k$-eccentricity of a vertex $v$ in a spanning tree as soon as the spanning tree is enumerated, and maintain the maximum Steiner $k$-eccentricity of $v$ over all currently enumerated spanning trees. The enumerating method is based on the following recursive equation. (Recall that if a graph $G$ is a tree itself, then one can invoke the linear time algorithm from~\cite{li2020steiner} to calculate the Steiner $k$-eccentricity of $v$.)

\begin{equation}\label{base}
 \ec_{k}(v,G)=\begin{cases}
                \ec_{k}(v,G); & G\ {\rm is\ a\ tree},\\
                \min\limits_{e\in E(C)}\{\ec_{k}(v,G-e)\};   & {\rm otherwise}\,,
              \end{cases}
\end{equation}
where $C$ is a cycle of $G$. The whole procedure is summarized in Algorithm~\ref{k-ecc}, where the parameter \emph{current-opt} is initialized to be zero and used to store the currently maximum Steiner $k$-eccentricity.

\medskip
\begin{algorithm}[H]\label{k-ecc}
\caption{k-ECC($v$, $G$, $k$, \emph{current-opt})}
\LinesNumbered 
\KwIn{Graph $G$, $v \in V(G)$,  integer $k\geq 3$. 
}
\KwOut{The Steiner $k$-eccentricity of $v$ in $G$.}
\If {$G$ is a tree}{\label{end-recur-start}
    \emph{temp} $\leftarrow$ Steiner $k$-eccentricity of $G$\;
    \If {\emph{temp} $<$ \emph{current-opt}}{
        \emph{current-opt} $\leftarrow$ \emph{temp}\:
    }
    \Return \emph{current-opt}\;
}\label{end-recure-done}
\Else{\label{start-recur}
    $C$ $\leftarrow$ simple cycle of $G$\;
    \For{each edge $e$ of $C$}{
        $H$ $\leftarrow$ $G-e$\;
        k-ECC($v$, $H$, $k$, \emph{current-opt})\;
        }
    }\label{stop-recur}
\end{algorithm}

\begin{theorem}\label{complexity}
Let $G$ be a connected graph and $v\in V(G)$. Then Algorithm~\ref{k-ecc} calculates the Steiner $k$-eccentricity of $v$ and can be implemented to run in $O(n^{\nu(G)+1}(n(G) + m(G) + k))$ time. 
\end{theorem}
 
\begin{proof}
The correctness of the algorithm follows from Theorem~\ref{ecc-on-graphs}. 

For the time complexity, we first note that in Step 2 we can apply the linear  algorithm from~\cite{li2020steiner} and that for Step 9 we can use the BFS algorithm starting from $v$. 
 
For the rest of the proof set $s = \nu(G)$, $n = n(G)$, and $m = m(G)$. Let $T(v,s)$ be the running time of Algorithm~\ref{k-ecc}. By~\eqref{base} we have
\begin{equation}\label{complexity-form}
T(v,s)=\begin{cases}
        O(kn); & s=0,\\
        O(n+m)+\ell_{1}*T(v,s-1); & s>0,
        \end{cases}
\end{equation}
where $\ell_{1} = m(C)$. Setting $M$=$n+m$, we can argue as follows: 
\begin{align}
    T(v,s) & = O(M)+\ell_{1}*T(v,s-1)\nonumber \\
           & = O(M)(1+\ell_{1}+\ell_{1}*\ell_{2}+\cdots +\Pi_{i=1}^{s}\ell_{i})+(\Pi_{i=1}^{s}\ell_{i})*T(v,0)\nonumber\\
           & = O(M)(1+O(n)+O(n)*O(n)+...+\Pi_{i=1}^{s}O(n))+(\Pi_{i=1}^{s}O(n))*O(kn)\nonumber\\
           & = O(n^{s+1}(n+m+k)), \nonumber
\end{align}
and we are done. 
\end{proof}

Note that the time complexity of Algorithm~\ref{k-ecc} grows exponentially with $\nu(G)$ rather than with $k$. 

\section{Linear time to calculate Steiner $3$-eccentricities for all vertices}
\label{all-ecc}

As already mentioned, in~\cite{li2020average} a linear time algorithm to calculate the Steiner $3$-eccentricity of a vertex of a tree was designed. In case one wishes to determine the Steiner $3$-eccentricity of all vertices, for instance in order to compute the average Steiner $3$-eccentricity, then this approach yields a quadratic algorithm. In this section we demonstrate that also the Steiner $3$-eccentricity of all vertices of a tree can be computed in linear time. Moreover, we also extend this result to weighted trees. 

Let $T$ and $T_{r}$, respectively, be a rooted weighted tree on $n$ vertices and the subtree rooted at a vertex $r \in V$. In other words, $T_{r}$ is a subgraph induced on vertex $r$ and all of its descendants.

The root of $T$ is assigned by an arbitrary vertex. The weights of edges are stored in an adjacency list named as $adj$. The linear algorithm is two-stage DFS procedures.  Details for the DFS algorithm can be found in~\cite{CoLeRiSt09}. The first stage is to compute the longest paths from $v$ in the subtree rooted at $v$ for every vertex $v$, while the second one is to update the longest paths with the upwards path via parent node for the purpose of computing Steiner 3-eccentricities. The first stage and the second stage are respectively showed in Algorithms \ref{first-DFS} and \ref{second-DFS}.

\medskip
\begin{algorithm}[H]
\label{first-DFS}
\KwIn{The adjacency matrix of the tree $T$ with the root vertex $root$.}
\KwOut{The downwards arrays $path\_weight, path\_index, attached\_weight$.}
\BlankLine

 $path\_weight [v] = (0, 0, 0)$\;
 $path\_index [v] = (-1, -1, -1)$\;
 $attached\_weight [v] = (0, 0, 0)$\;
 \For{every neighbor $u$ of $v$} {
     \If {$parent[u] = -1$ and $u \neq root$} {
            $parent [u] = v$\;
            $DFS\_stage1 (u)$\;
            $update (v, u, adj\_weight [v][u] + path\_weight [0][u],
                 \max (path\_weight [1][u], attached\_weight [0][u]))$\;
     }
 }
\caption{DFS$\_$stage1 ($v$)}
\end{algorithm}

\medskip
\begin{algorithm}[H]
\label{second-DFS}
\caption{DFS$\_$stage2 ($v$)}

    \KwIn{The outputs of $DFS\_stage1$.}
    \KwOut{The arrays $path\_weight, path\_index, attached\_weight$.}
    \BlankLine

    $mark [v] = 1$\;
    $u = parent [v]$\;
    \If{$u \neq -1$} {
        $up\_path\_weight = 0$\;
        $up\_attached\_weight = 0$\;
        \If{$path\_index [0][u] \neq v$} {
            $up\_path\_weight = adj\_weight [u][v] + path\_weight [0][u]$\;
            \If{$path\_index [1][u] \neq v$} {
                $up\_attached\_weight = \max(path\_weight [1][u], attached\_weight [0][u])$\;
            }\Else{
                $up\_attached\_weight = \max(path\_weight [2][u], attached\_weight [0][u])$\;
            }
        }\Else{
            $up\_path\_weight = adj\_weight [u][v] + path\_weight [1][u]$\;
            $up\_attached\_weight = \max(path\_weight [2][u], attached\_weight [1][u])$\;
        }
        $update (v, u, up\_path\_weight, up\_attached\_weight)$\;
    }

   \For{every neighbor $u$ of $v$} {
     \If {$mark[u] = -1$} {
          $DFS\_stage2 (u)$\;
     }
 }

\end{algorithm}

\bigskip
The Steiner 3-eccentricity of the vertex $v$ can be computed as
$$
\epsilon_3(v) = path\_weight [v][0] + \max\{ path\_weight [v][1], attached\_weight [v][0] \},
$$
where $path\_weight [v]$ is initialized as  $(0, 0, 0)$ and represents the length of the longest path from the vertex $v$ in the subtree $T_v$; $path\_index [v]$ is initialized as $(-1, -1, -1)$ and represents the neighbor of $v$ on the longest path from the vertex $v$ in the subtree $T_v$; $attached\_weight [v]$ is initialized as $(0, 0, 0)$ and represents the length of the longest subpath attached at the longest path in the subtree $T_v$ strictly bellow $v$.

The main helper function
$update(v, u, new\_weight, new\_attached\_weight)$ is to keep the top three values for the longest paths coming from the vertex $v$ and implemented by the C code in Appendix.

The first stage $DFS\_stage1$ is to recursively traverse all neighbors of the vertex $v$. After the subtree is processed, we update the values for the node $v$ based on the values for each of the subtrees. We traverse the nodes in the pre-order phase in the second stage $DFS\_stage2$. We update the values for the root and then run computation for the neighbors.

\begin{theorem}\label{crrrectness}
There is a linear-time algorithm to calculate the Steiner $3$-eccentricities for all vertices in a weighted tree $T$.
\end{theorem}

\begin{proof}
Given that the algorithm consists of two traversals using DFS algorithms and with linear initialization, the time complexity of the algorithm equals $O(n)$. There are six additional vectors of size $n$, and therefore the memory complexity of the algorithm equals $O(n)$ as well.

The correctness of the algorithm directly follows from the definition of Steiner 3-eccentricities for the root vertex $v$. For other vertices, we effectively compute the top three longest paths and attached paths - which is equivalent as considering those vertices as roots.
\end{proof}

\section{Conclusion}\label{conclusions}

We presented a linear-time algorithm to solve the $k$-ST problem on block graphs, and an algorithm to solve the $k$-ST problem on general graphs, where the exponential growth of the running time depends only on the cyclomatic number of a graph. 

It seems that if a graph is dense, then the Steiner $k$-eccentricity of a vertex may be easy to find. For instance, this is the case for complete graphs and for complete graphs with one or two edges removed. Inspired by this, an open question is how to modify Algorithm \ref{k-ecc} so that it works well not only for small values of $\nu$ but also when $\nu$ is large. 

For calculating the Steiner $k$-eccentricity of a vertex in a graph $G$, we showed that there is an algorithm  whose running time grows exponentially on $\mu$ but is independent of the input parameter $k$. On the other hand, is there is a fixed-parameter tractable algorithm \cite{Cygan2015ParameterizedA} to solve this problem? Moreover, there is still no answer to the question of whether the problem is NP-hard.

Let's turn back to the $k$-ST problem. One can also ask whether there is a $k$-set $S$ such that the size of the minimum Steiner tree on $S$ is at least $c$. This yields the \emph{Steiner $k$-eccentricity set problem} ($k$-SES). The decision version of the \emph{Steiner k-eccentricity set problem} ($k$-SES) is presented in Table~\ref{k-SESP}. 

\begin{table}[ht!]
	\centering  
	\caption{The Steiner k-eccentricity set problem ($k$-SES)}  
	\label{k-SESP}  
	\begin{tabular}{|l|}
		\hline
		Instance: Graph $G$, $v\in V(G)$, $k\in [n(G)]$, constant $c$.\\
		\hline
        Question: Is there  a $k$-set $S$, where $v\in S$, such that the size of a minimum \\
        \qquad\qquad\ \ Steiner tree on $S$ is at least $c$? \\
        \hline
	\end{tabular}
\end{table}

For every vertex $v$ in a graph $G$, there is a $k$-set $S_{1}$ such that a minimum Steiner tree on $S_{1}$ has at least $c$ edges if and only if there is a minimum Steiner tree on some $k$-set $S_{2}$ such that the size of  the Steiner tree is at least $c$, where $v\in S_{1}$ and $v\in S_{2}$. In other words, there is a "YES" answer to the $k$-ST problem if and only if there is a "YES" answer to the $k$-SES problem. Therefore,  the $k$-ST problem is as hard as the $k$-SES problem. However, algorithms to solve these two problems could be different. It seems that there is no brute-force algorithm to solve the $k$-SES problem.

Finally, we designed an $O(n)$-time algorithm to calculate Steiner $3$-eccentricities for all vertices of a tree. Does this result extend to $k\ge 4$?

\section*{Acknowledgements}

This work was supported by  Science Foundation of Guizhou University of Finance and Economics (2020XYB16), Science Foundation of Guizhou University of Finance and Economics (2019YJ058). Sandi Klav\v{z}ar acknowledges the financial support from the Slovenian Research Agency (research core funding P1-0297, and projects N1-0095, J1-1693, J1-2452).

\end{document}